\documentclass[a4paper,12pt, reqno]{amsart}
\usepackage{amsmath,amssymb,enumerate,epsfig,xcolor}

\usepackage{graphicx}
\usepackage{import}
\usepackage{xifthen}
\usepackage{pdfpages}
\usepackage{transparent}

\usepackage{enumitem}
\usepackage{comment}

\usepackage{float}

\newenvironment{pf}[1][Proof.]{\noindent \textbf{#1:} }{}

\newtheorem{thm}{Theorem}

\newtheorem{prop}[thm]{Proposition}
\newtheorem{lemma}[thm]{Lemma}
\newtheorem{claim}{Claim}

\newtheorem{q}[thm]{Question}

\newtheorem{Rmks}[thm]{Remarks}

\theoremstyle{definition}
\newtheorem{definition}[thm]{Definition}
\newtheorem*{definition*}{Definition}
\newtheorem{exple}[thm]{Example}
\newtheorem*{exple*}{Example}
\newtheorem{remark}[thm]{Remark}
\newtheorem*{rmk}{Remark}
\newtheorem*{rmks}{Remarks}

\newcommand\xqed[1]{\leavevmode\unskip\penalty9999 \hbox{}\nobreak\hfill
	\quad\hbox{#1}}
\newcommand\demo{\xqed{$\Diamond$}}

\newcommand{\N}{\mathbb{N}}

\newcommand{\R}{\mathbb{R}}
\newcommand\Om{\Omega}
\newcommand\om{\omega}
\newcommand\x{\times}

\renewcommand\phi{\varphi}

\newcommand\eps{\varepsilon}

\newcommand\wo{\setminus}

\newcommand\nn{{\nonumber}}
\newcommand\sub{\subseteq}

\newcommand\smfld{(M, \omega)}
\newcommand\smfldp{(M', \omega')}
\newcommand\omst{\omega_{\mathrm{st}}}
\newcommand{\Forms}[2]{\Omega^{#1, #2}}
\newcommand\Int{\operatorname{Int}}

\title{Recognition of objects through symplectic capacities}
\author{Yann Guggisberg\and Fabian Ziltener}
\address{Affiliation of Y.~Guggisberg: Utrecht University\\
mathematics institute\\
Hans Freudenthalgebouw\\
Budapestlaan 6\\
3584 CD Utrecht\\
The Netherlands
}
\email{y.b.guggisberg@uu.nl}
\address{Affiliation of F.~Ziltener: Utrecht University\\
mathematics institute\\
Hans Freudenthalgebouw\\
Budapestlaan 6\\
3584 CD Utrecht\\
The Netherlands
}
\email{f.ziltener@uu.nl}
\thanks{Y.~Guggisberg's work on this publication is part of the project \emph{Symplectic capacities, recognition, discontinuity, and helicity} (project number 613.009.140) of the research programme \emph{Mathematics Clusters}, which is financed by the Dutch Research Council (NWO). We gratefully acknowledge this funding.}
  
\begin{document}

\maketitle

\begin{abstract}We prove that the generalized symplectic capacities recognize objects in symplectic categories whose objects are of the form $(M, \omega)$, such that $M$ is a compact and 1-connected manifold, $\omega$ is an exact symplectic form on $M$, and there exists a boundary component of $M$ with negative helicity. The set of generalized symplectic capacities is thus a complete invariant for such categories. This answers a question by Cieliebak, Hofer, Latschev, and Schlenk. It appears to be the first result concerning this question, except for recognition results for manifolds of dimension 2, ellipsoids, and polydiscs in $\mathbb{R}^4$. Strikingly, our result holds more generally for differential form categories. Recognition of objects is therefore not a symplectic phenomenon.

We also prove a version of the result for normalized capacities.
\end{abstract}

\tableofcontents

\section{Introduction and main results}

The main results of this article are concerned with recognition through generalized capacities and normalized capacities on differential form categories. These generalize (generalized and normalized) symplectic capacities and symplectic categories. To explain our results, we define the notion of a capacity in the following even more general setup. Let $\widehat{\mathcal{C}}$ be a category determined by a formula. By this, we mean a category whose objects, morphisms, and pairs constituting the composition map\footnote{By these we mean pairs $((f, g), h)$, where $(f, g)$ is the input of the composition map and $h$ is the output, i.e. $f \circ g = h$.} are determined by well-formed logical formulas.
\begin{rmk}[categories and set theory] This article is based on the Zermelo-Fraenkel axiomatic system with choice (ZFC). In most cases, the collection of objects and morphisms of a category are proper classes and not sets. Therefore, we need to be careful when talking about general categories. However, categories determined by a formula can be dealt with in ZFC.\demo
\end{rmk}
We assume that $\widehat{\mathcal{C}}$ is locally small\footnote{This means that for every pair of objects $(A, B)$ of $\widehat{\mathcal{C}}$, the collection of morphisms from $A$ to $B$ is a set.}. Suppose also that there exists a set $S$ of objects of $\widehat{\mathcal{C}}$, such that every object of $\widehat{\mathcal{C}}$ is isomorphic to an element of $S$. We denote $\R^+:=(0, \infty)$ and fix an $\R^+$-action on $\widehat{\mathcal{C}}$ that is defined by a logical formula. (See Appendix \ref{section_subcat}, Definition \ref{def_action_cat}.) Let $\mathcal{C}$ be an $(\R^+)$-invariant isomorphism-closed subcategory of $\widehat{\mathcal{C}}$. (See Appendix \ref{section_subcat}, Definitions \ref{def_isom_closed_subcat} and \ref{def_invariant_subcat}.) \label{setting}

\begin{exple*}[invariant isomorphism-closed subcategory] As an example, $\mathcal{C}=\widehat{\mathcal{C}}$ is an $\R^+$-invariant, isomorphism-closed subcategory of itself. This follows from the assumptions of local smallness and the existence of a set $S$ as above. (See Appendix \ref{section_subcat}, Remark \ref{rmk_isom_closed_subcat}.)
\end{exple*}

\begin{definition}[generalized and normalized capacity]\label{def_capacity}
	A \emph{(generalized) capacity}\footnote{In \cite{CHLS07}, a \emph{capacity} is defined to be a generalized capacity that is non-trivial. Since we do not use that condition in this article, we will use the word \emph{capacity} in the sense of \emph{generalized capacity}.} on $\mathcal{C}$ is a pair $(\mathcal{C}_0, c)$, where:
	\begin{itemize}
		\item $(\mathcal{O}_0, \mathcal{M}_0) := \mathcal{C}_0$ is an invariant full small subcategory of $\mathcal{C}$, such that every object of $\mathcal{C}$ is isomorphic to an object of $\mathcal{C}_0$.
		\item $c:\mathcal{O}_0 \to [0, \infty]$ is a map with the following properties:
		\begin{enumerate}[label = (\roman*)]
			\item \textbf{(monotonicity)} If  $A$ and $B$ are two objects in $\mathcal{O}_0$, such that there exists a $\mathcal{C}_0$-morphism from $A$ to $B$, then
			\begin{equation*}
				c(A) \leq c(B).
			\end{equation*}
			\item \textbf{(conformality)} For every $A \in \mathcal{O}_0$ and every $a \in \R^+$ we have
			\begin{equation*}
				c(a_*A) = ac(A),
			\end{equation*}
		where $a_*A$ denotes the action of $a \in \R^+$ on the object $A$.
		\end{enumerate}
	Let $X$ be a set of objects of $\mathcal{C}$. We say that a capacity $(\mathcal{C}_0, c)$ is ($X$-)\emph{normalized} iff
	\begin{enumerate}[resume*]
		\item \textbf{(normalization)}\label{condition_normalization} $c(A) = 1$ for every $A$ in $\mathcal{O}_0$ that is isomorphic to an element of $X$.
	\end{enumerate}
	\end{itemize}
\end{definition}

In the following $\mathcal{O}_0$ always denotes the set of objects of $\mathcal{C}_0$.

\begin{rmks}[capacity]
	\begin{itemize}
		\item It follows from Definition \ref{def_isom_closed_subcat}, the local smallness of $\widehat{\mathcal{C}}$, and Remark \ref{rmk_invariant_subcat} that there exists an invariant full small subcategory $\mathcal{C}_0$ of $\mathcal{C}$, such that every object of $\mathcal{C}$ is isomorphic to an object of $\mathcal{C}_0$. For every such $\mathcal{C}_0$, the pair $(\mathcal{C}_0, c=0)$ is a capacity. Hence, capacities on $\mathcal{C}$ exist.
		\item  It is not possible to define the class of all maps between two classes in ZFC, even if the target class is actually a set. For this reason, we restrict the domain of a capacity to a set of objects $\mathcal{O}_0$ as above.
		\item If $A$ and $B$ are isomorphic objects in $\mathcal{C}$, then by monotonicity, every capacity $(\mathcal{C}_0, c)$, such that $A$ and $B$ lie in $\mathcal{O}_0$, satisfies $c(A) = c(B)$.
	\end{itemize}
	\demo
\end{rmks}

\begin{definition*}[recognition]
	\begin{itemize}
		\item Let $\mathcal{C}$ be an invariant isomorphism-closed subcategory of $\widehat{\mathcal{C}}$. We say that the (generalized) capacities on $\mathcal{C}$ \emph{recognize objects} iff for every pair $(A, B)$ of non-isomorphic objects of $\mathcal{C}$ there exists a capacity $(\mathcal{C}_0, c)$, such that $A, B \in \mathcal{O}_0$ and $c(A) \neq c(B)$.
		\item Let $X$ be a set of objects of $\mathcal{C}$. We say that the $X$-normalized capacities on $\mathcal{C}$ \emph{recognize objects} iff for every pair $A, B$ of non-isomorphic objects of $\mathcal{C}$, such that $A$ or $B$ is not isomorphic to any element of $X$, there exists an $X$-normalized capacity $(\mathcal{C}_0, c)$, such that $A, B \in \mathcal{O}_0$ and $c(A) \neq c(B)$.
\end{itemize}
\end{definition*}
\begin{rmk}[recognition]
	By the normalization condition \ref{condition_normalization} in Definition \ref{def_capacity}, there is no $X$-normalized capacity that distinguishes between two objects that are isomorphic to elements of $X$.
	\demo
\end{rmk}

In order to define the notion of a form category, we need the following. Let $m, k \in\N_0:=\{0,1,\ldots\}$.
\begin{definition*}[universal form category] We define $\Forms{m}{k}$, the \emph{universal $(m,k)$-form category} as follows:
\begin{itemize}
\item Its objects are the pairs $(M, \omega)$, where  $M$ is a (smooth) manifold\footnote{All the manifolds in this article are smooth, finite-dimensional, and are allowed to have boundary.} of dimension $m$ and $\omega$ is a differential $k$-form on $M$.
\item Its morphisms are (smooth) embeddings\footnote{We do not impose any condition involving the boundaries of the manifolds.} that intertwine the differential forms.
\end{itemize}
\end{definition*}

\begin{rmk}[universal form category] The category $\Forms{m}{k}$ is determined by a logical formula and locally small. For every $a \in \R^+$, we consider the functor that sends an object $(M, \omega)$ to $(M, a\omega)$ and a morphism $\varphi:(M, \omega) \to (M', \omega')$ to the same map viewed as a morphism $(M, a\omega) \to (M', a\omega')$. This defines an $\R^+$-action on $\Forms{m}{k}$.
	
	For every set $X$, we denote by $\mathcal{P}(X)$ its power set. Every object of $\Forms{m}{k}$ is isomorphic to an element of the set
	\begin{align}\label{eq_subcat_beth1}
		\begin{split}
		\mathcal{O}^{m, k}_0:=\big\{&(M, \omega) \text{  object of }\Forms{m}{k}\,\big|\,\\
		&\text{The set underlying }M\text{ is a subset of }\mathcal{P}(\N_0).\big\}.
	\end{split}
	\end{align}
	 This is explained in \cite[p.12, footnote 28]{JZ21}. Therefore, $\widehat{\mathcal{C}}:=\Forms{m}{k}$ fits into the setting described on page \pageref{setting} of this article.
	 \demo
\end{rmk}

\begin{definition*}[$(m, k)$-form category, symplectic category and capacity]
	An $(m, k)$-\emph{(differential) form category} is an invariant isomorphism-closed subcategory of $\Forms{m}{k}$. For every $n\in\N_0$, we call a $(2n, 2)$-form category whose objects are symplectic manifolds, a \emph{symplectic category}. We call a (generalized) capacity on a symplectic category also a \emph{(generalized) symplectic capacity}.
\end{definition*}

\begin{exple}[embedding capacity]\label{exple_embedding_capacity}
Let $\mathcal{C} = (\mathcal{O}, \mathcal{M})$ be an $(m, k)$-form category, $\mathcal{C}_0$ be the invariant full small subcategory of $\mathcal{C}$ whose set of objects is given by $\mathcal{O}_0:= \mathcal{O}\cap\mathcal{O}^{m, k}_0$, and let $(M, \omega)$ be an object of $\Forms{m}{k}$. We define the \emph{(domain-)embedding capacity} for $\smfld$ and $\mathcal{C}$ to be the map
\begin{align}
\nonumber c_{\smfld} := c_{\smfld}^{\mathcal{C}}&: \mathcal{O}_0 \to [0, \infty],\\ 
\begin{split}\label{eqn_def_embedding_capacity}
c_{\smfld}\smfldp :=\\
\sup \big\{ a\in (0, \infty) &\,\big|\, \exists \,\Forms{m}{k}\text{-morphism }(M, a\omega) \to \smfldp  \big\}.
\end{split}
\end{align}
The pair $(\mathcal{C}_0, c_{(M, \omega)})$ is a capacity.

Let $n\in\N_0$. We denote by $B := B^{2n}$ the $2n$-dimensional open unit ball, by $Z:=Z^{2n}:=B^2\times \R^{2n-2}$ the open unit cylinder, and by $\omega_{\mathrm{st}}$ the standard symplectic form on $\R^{2n}$. Consider the case $k=2$. We define the \emph{Gromov width} for $\mathcal{C}$ to be
\begin{equation*}
w := w^{\mathcal{C}} := c_{(B, \omst|_B)}^{\mathcal{C}}.
\end{equation*}
If $B$ and $Z$ are objects of $\mathcal{C}$, then $(\mathcal{C}_0, w)$ is a $\{B, Z\}$-normalized capacity by Gromov's non-squeezing Theorem \cite{Gro85}.
\end{exple}

The goal of this article is to provide an answer to the following question by K.~Cieliebak, H.~Hofer, J.~Latschev, and F.~Schlenk for certain symplectic categories containing many objects.

\begin{q}[\cite{CHLS07}, question 2, p. 20]\label{q_recognition}
Do the generalized symplectic capacities on a given symplectic category recognize objects?
\end{q}

This question appears to be widely open. To our knowledge, it has only been answered in dimension 2, for the category of ellipsoids in $\R^{2n}$ (for every $n \in \N$), and for polydiscs in $\R^4$. More precisely, a Moser-type argument classifies closed\footnote{This means compact and without boundary.} connected symplectic manifolds of dimension 2 by their total area. The Ekeland-Hofer capacities recognize ellipsoids in $\mathbb{R}^{2n}$. (See \cite[Fact 10, p.27]{CHLS07}.) The Gromov width and the volume capacity recognize polydiscs in $\R^4$. (See \cite[Example 6(iv), p.21]{CHLS07}.)

A negative answer to a variant of Question \ref{q_recognition} was recently provided by E.~Kerman and Y.~Liang in \cite[Theorem 1.3]{KL}. These authors showed that the Gutt-Hutchings capacities and the volume capacity do not recognize starshaped compact smooth submanifolds (with boundary) in $\R^4$. (This also follows from the proof of Proposition \ref{prop:neg hel} below.)

To state our main results, we need the following definitions. Let $k, n \in \N:=\{1,2,\ldots\}$.

\begin{definition}[helicity]\label{def_helicity}
	Assume that $n\geq2$. Let $N$ be a closed $(kn -1)$-dimensional manifold with an orientation $O$, and let $\sigma$ be an exact $k$-form on $N$. We define the helicity of $(N, O, \sigma)$ to be the integral
	\begin{equation*}
		h(N, O, \sigma) = \int_{N, O} \alpha \wedge \sigma^{\wedge(n-1)},
	\end{equation*}
	where $\alpha$ is an arbitrary primitive of $\sigma$ and $\int_{N, O}$ denotes the integral on $N$ with respect to the orientation $O$. By Lemma \ref{lemma_helicity_well_def} below, this number is well-defined. We detail some basic properties of helicity in Section \ref{section_helicity} below.
\end{definition}

\begin{definition*}[maxipotent]
	A $k$-form $\om$ on a $kn$-dimensional manifold $M$ is called \emph{maxipotent} iff $\omega^{\wedge n}$ is nowhere-vanishing.
\end{definition*}
\begin{rmk}[maxipotent]
	\begin{itemize} 
		\item If $n \geq 2$ and $k$ is odd, we have that $\omega \wedge \omega =0$, and hence there are no maxipotent $k$-forms.
		\item A maxipotent $k$-form $\omega$ on a $kn$-dimensional manifold induces an orientation on the manifold. The orientation is induced by the $kn$-form $\omega^{\wedge n}$, which is nowhere-vanishing, and hence a volume form.
	\end{itemize}
\demo
\end{rmk}

The first main result of this article is as follows.

\begin{thm}[recognition of objects]\label{thm_recognition}
Assume that $n\geq2$. Let $\mathcal{C}$ be a $(kn, k)$-form category, such that for every object $(M, \omega)$ of $\mathcal{C}$, the manifold $M$ is compact and 1-connected\footnote{This means connected and simply connected.}, $\omega$ is exact and maxipotent, and $\om$ induces negative helicity on some connected component of the boundary $\partial M$ \footnote{By this we mean that the pullback of $\omega$ to the component has (strictly) negative helicity with respect to the orientation induced by $\omega$.}. Then the generalized capacities on $\mathcal{C}$ recognize objects.
\end{thm}

This result answers Question \ref{q_recognition} affirmatively for every symplectic category satisfying the hypotheses of Theorem \ref{thm_recognition}. To our knowledge, this is the first result concerning Question \ref{q_recognition}, except for the results mentioned above. Note that Theorem \ref{thm_recognition} treats the more general setting of $(kn,k)$-form categories. This shows that recognition of objects by capacities is not a symplectic phenomenon.
\begin{rmk}[even degree] Under the hypotheses of Theorem \ref{thm_recognition}, the integer $k$ is even, provided that $\mathcal{C}$ is nonempty. This follows from the conditions that $n\geq2$ and $\om$ is maxipotent.
\end{rmk}
The compactness and the helicity hypotheses of Theorem \ref{thm_recognition} cannot be dropped. See Proposition \ref{prop:neg hel} and Remark \ref{rmk:cpt} below.

\begin{exple*}[recognition of objects] A category with a large class of objects to which Theorem \ref{thm_recognition} applies, is provided by the following full subcategory $\mathcal{C}=(\mathcal{O}, \mathcal{M})$ of $\Forms{kn}{k}$. Let $M$ be a compact $kn$-dimensional manifold, $\omega$ an exact maxipotent $k$-form on $M$, and $U \subseteq M$ a nonempty, open subset of $M$ whose boundary is a smooth submanifold of the interior of $M$ and lies on one side of $U$. The class $\mathcal{O}$ consisits of all $(M\setminus U, \omega|_{M\setminus U})$ obtained in this way, such that $M \setminus U$ is 1-connected. (The objects satisfy the negative helicity assumption by Stokes' Theorem for helicity, Lemma \ref{lemma_stokes_helicity} below.) Hence $\mathcal{C}$ satisfies the hypotheses of Theorem \ref{thm_recognition}. If $k=2$ then $\mathcal{O}$ consists of symplectic manifolds, and hence $\mathcal{C}$ is a symplectic category.
\end{exple*}

\begin{rmk}[capacities as a complete set of invariants] Let $\widehat{\mathcal{C}}$ be a locally small category determined by a formula and such that there exists a set of objects of $\widehat{\mathcal{C}}$ with the property that every object of $\widehat{\mathcal{C}}$ is isomorphic to an element of that set. Let $G$ be a group that is determined by a formula, and let $\rho$ denote a $G$-action on $\widehat{\mathcal{C}}$ that is determined by a formula. Let $\mathcal{C}$ be an invariant isomorphism-closed subcategory of $\widehat{\mathcal{C}}$.
	
	By a \emph{complete invariant} for $(\mathcal{C}, \rho)$, we mean a triple $(\mathcal{C}_0, \mathcal{C}_1, F)$, where $\mathcal{C}_0$ is an invariant full small subcategory of $\mathcal{C}$, such that every object of $\mathcal{C}$ is isomorphic to an object of $\mathcal{C}_0$, $\mathcal{C}_1$ is a small category equipped with a $G$-action determined by a formula, and $F: \mathcal{C}_0 \to \mathcal{C}_1$ is an equivariant and essentially injective\footnote{This means that for each pair of objects $A, B$ of $\mathcal{C}_0$ satisfying $F(A) \cong F(B)$, we have $A \cong B$.} functor. For two sets $X$ and $Y$, we denote by $Y^X$ the set of maps from $X$ to $Y$.

By a \emph{complete set of invariants} for $(\mathcal{C}, \rho)$, we mean a triple $(\mathcal{C}_0, \mathcal{C}_1, \mathcal{F})$, where $\mathcal{C}_0$ and $\mathcal{C}_1$ are as above, and $\mathcal{F}$ is a set of equivariant functors $\mathcal{C}_0 \to \mathcal{C}_1$, such that the following pair of maps is a complete invariant for $(\mathcal{C}, \rho)$:
	\begin{align*}
		\mathcal{C}_0 &\to \mathcal{C}_1^{\mathcal{F}}:=(\mathcal{O}_1^{\mathcal{F}}, \mathcal{M}_1^{\mathcal{F}}),\\
		\mathcal{O}_0 \ni A &\mapsto \big(F(A)\big)_{F \in \mathcal{F}} \in \mathcal{O}_1^{\mathcal{F}},\\
		\mathcal{M}_0 \ni (A, B, \varphi) &\mapsto \big(F(A), F(B), F(\varphi)\big)_{F \in \mathcal{F}} \in \mathcal{M}_1^{\mathcal{F}}.
	\end{align*}
	
	Now, we assume that $\mathcal{C}$ is a $(kn, k)$-form category, with $n \geq 2$ and $k$ even, and $\mathcal{C}_0$ is an invariant full small subcategory of $\mathcal{C}$, such that every object of $\mathcal{C}$ is isomorphic to an object of $\mathcal{C}_0$. We choose $\mathcal{C}_1$ to be the small category whose objects are the elements of $[0, \infty]$ and whose morphism set between $a, b \in [0, \infty]$ is $\{(a, b)\}$ iff $a \leq b$ and $\emptyset$ otherwise, and we choose $\mathcal{F}$ to be the set of capacities of the form $(\mathcal{C}_0, c)$, viewed as functors. We denote by $\rho$ the restriction of the $\R^+$-action on $\Forms{kn}{k}$ to $\mathcal{C}$. Theorem \ref{thm_recognition} means that $(\mathcal{C}_0, \mathcal{C}_1, \mathcal{F})$ is a complete set of invariants for $(\mathcal{C},\rho)$.
\demo
\end{rmk}

The idea of the proof of Theorem \ref{thm_recognition} is the following. Let $(M, \omega)$ and $(M', \omega')$ be two objects of $\mathcal{C}$ and assume that all capacities whose domain contains $\smfld$ and $\smfldp$ agree on $\smfld$ and $\smfldp$. Since the embedding capacities $c_{\smfld}$ and $c_{\smfldp}$ agree on $\smfld$ and $\smfldp$, there exist embeddings
\begin{equation*}
	 \varphi: (M, a\omega) \hookrightarrow \smfldp \text{ and } \psi: (M', b\omega') \hookrightarrow \smfld
\end{equation*}
for $a, b$ in $(0, 1]$ that are arbitrarily close to 1. The composition $\chi := \psi \circ \varphi$ then yields an embedding $(M, C\omega) \hookrightarrow \smfld$ with $C = ab$. 

Our assumption that $M$ has some negative helicity boundary component, Stokes' Theorem for helicity (Lemma \ref{lemma_stokes_helicity} below), and the fact that $C=ab$ is close to 1 imply that the constant $C$ is equal to 1.\footnote{More precisely, these facts imply the following. We denote by $I$ the set of connected components of the boundary of $M$. The embedding $\chi$ induces a partition of the set $I\sqcup I$. Two elements of $I\sqcup I$ lie in the same partition element iff they lie in the same path component of $M \setminus \chi(M)$. Here we identify each element of the first component of $I\sqcup I$ with its image under $\chi$. 

Every element of the partition containing a negative helicity component does not contain any positive helicity component. (See the Key Lemma \ref{lemma_positive_helicity_partition}.) Using Stokes' Theorem for helicity again, it follows that $C=1$.} This implies that $a = b =1$, which implies that $\varphi$ and $\psi$ are isomorphisms. This finishes the outline of the proof of Theorem \ref{thm_recognition}. Figure \ref{fig_embeddings} illustrates the possible kinds of embeddings of $(M, C\omega)$ into $(M, \omega)$ for spherical shells, and provides a pictorial idea of the main step of the proof.\\

\begin{figure}[ht]
	\centering
	
	\def\svgwidth{0.8\columnwidth}
	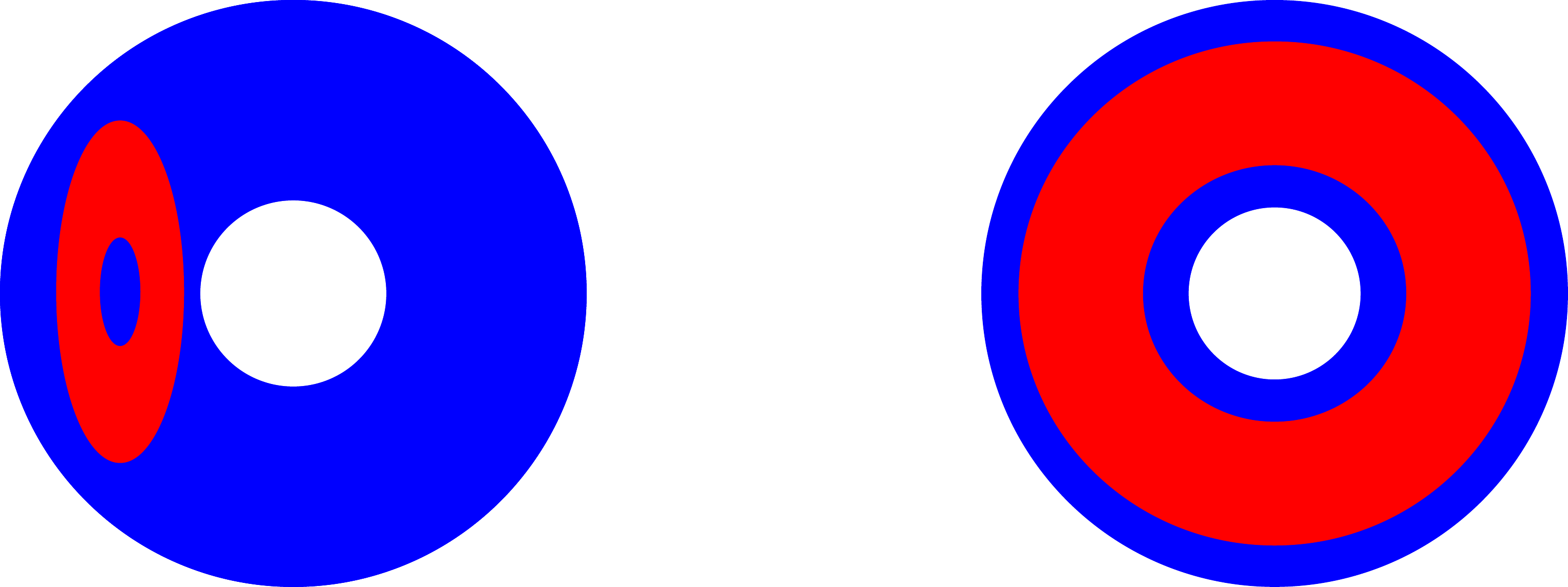

	\caption{The two blue annuli depict the target spherical shell $M$ and the two red annuli the images of $M$ under two form-intertwining (i.e., symplectic, if $k=2$) embeddings $\chi = \psi \circ \varphi$. The embedding on the left only exists for small values of $C$, as the hole in the middle of $\chi(M)$ ``takes space''. The embedding on the right wraps the interior hole of the red shell around the interior hole of the blue shell. Using the equality $\chi^*\om=C^n\om$ and Stokes' Theorem for helicity, the volume of the thin blue shell around the central hole is $C^n -1 \geq 0$. On the other hand, by assumption, we have $C \leq 1$, which implies that $C=1$.}
	\label{fig_embeddings}
\end{figure}

Let $n\geq 2$ and $\mathcal{C}=(\mathcal{O},\mathcal{M})$ be a $(2n,2)$-form category. Recall that $B$ and $Z$ denote the open unit ball and cylinder. From this point on, by a \emph{normalized capacity}, we mean a $\{B, Z\}$-\emph{normalized} capacity. Our second main result states that the normalized capacities recognize objects under the hypotheses of Theorem \ref{thm_recognition} and one more assumption on $\mathcal{C}$:

\begin{prop}[recognition of objects through normalized capacities]\label{prop_recognition_normalized_capacities}
	 Assume that $B, Z \in \mathcal{O}$. Suppose also that every object $(M,\omega)\in\mathcal{O}$ satisfies the assumptions of Theorem \ref{thm_recognition} and the inequality
	\begin{equation}\label{ineq_thin_object}
		w(M, \omega) c_{(M, \omega)}(Z) < 1,\footnotemark
\end{equation}
\footnotetext{\label{evaluate}Recall the definitions of the embedding capacity $c_{(M, \omega)}$ and the Gromov width $w$ from Example \ref{exple_embedding_capacity}. Here we evaluate $c_{(M, \omega)}$ on a general symplectic manifold $(M',\om')$ by evaluating it on an object that is isomorphic to $(M',\om')$ and lies in the set $\mathcal{O}_0^{m, k}$, defined in \eqref{eq_subcat_beth1}.}
except if $(M,\om)$ is isomorphic to $B$ or $Z$. Then the normalized capacities recognize objects in $\mathcal{C}$.
\end{prop}
\begin{rmks}[recognition of objects through normalized capacities]
\begin{itemize}
\item Examples of manifolds for which assumption \eqref{ineq_thin_object} holds are thin spherical shells, as explained in \cite[Corollary 5]{SZ12}.
\item The inequality \eqref{ineq_thin_object} is invariant in the sense that if it holds for an object $(M, \omega)$ in $\Forms{2n}{2}$ then it also holds for $(M, a\omega)$, for every constant $a \in (0, \infty)$.
\item By definition, no normalized capacity distinguishes between $B$ and $Z$. The result above states that the normalized capacities distinguish between all other objects and every object in the category.
\end{itemize}
\demo
\end{rmks}
The idea of the proof of Proposition \ref{prop_recognition_normalized_capacities} is similar to that of Theorem \ref{thm_recognition}. Here we replace the domain-embedding capacity $c_{(M, \omega)}$ by the maximum of a rescaling of $c_{(M, \omega)}$ and the Gromov width.

\subsection*{Organization of the article}
In Sections \ref{section_proof_thm} and \ref{section_proof_prop}, we prove the main results. In Section \ref{sec:nec} we show that the compactness and the negative helicity hypotheses cannot be dropped. Appendix \ref{section_subcat} deals with the set- and category-theoretic setting needed to defined capacities. In Appendix \ref{section_helicity} we go over some basic properties of helicity, such as Stokes' theorem for helicity. These properties are indirectly used in the proofs of the main results.

\subsection*{Acknowledgments} We would like to thank Vuka\v{s}in Stojisavljevi\'c for making us aware of \cite[Theorem 1.1]{EH96}, on which Proposition \ref{prop:neg hel} below is based. We thank Jaap van Oosten for a discussion about categories and ZFC. We would like to thank Felix Schlenk for an interesting correspondance.

\section{Proof of Theorem \ref{thm_recognition}}\label{section_proof_thm}

Let $k, n\in \N$ with $n\geq 2$. We will prove Theorem \ref{thm_recognition} by reducing it to the following lemma.
\begin{lemma}[sufficient criterion for recognition]\label{lemma_isomorphism} Let $M$ be a compact, 1-connected, $kn$-dimensional manifold and $\om$ an exact, maxipotent $k$-form on $M$, such that $\om$ induces negative helicity on some connected component of $\partial M$. Let $(M',\om')$ be another object with these properties. Assume that
\begin{equation}\label{eqn_embedding_capacity_isom}
c_{(M, \omega)}(M', \omega') \geq 1 \quad \text{ and } \quad c_{(M', \omega')}(M, \omega) \geq 1.\footnote{Compare to footnote \ref{evaluate}.}
\end{equation}
Then there is a diffeomorphism $\varphi: M \to M'$, such that $\varphi^*\omega' = \omega$.
\end{lemma}
For a proof, see page \pageref{proof_lemma_isomorphism}.

\begin{proof}[Proof of Theorem \ref{thm_recognition}]
	Assume that $(M, \omega)$ and $(M', \omega')$ are two objects of $\mathcal{C}$, such that all capacities whose domains contain $(M, \omega)$ and $(M', \omega')$, agree on $(M, \omega)$ and $(M', \omega')$. We show that $(M, \omega)$ and $(M', \omega')$ are isomorphic. It follows from our assumption that
	\begin{align*}
		& c_{(M, \omega)}(M', \omega') = c_{(M, \omega)}(M, \omega),\\
		& c_{(M', \omega')}(M, \omega) = c_{(M', \omega')}(M', \omega').\end{align*}

	Since $c_{(M, \omega)}(M, \omega) \geq 1$ and $c_{(M', \omega')}(M', \omega') \geq 1$, it follows that the hypothesis \eqref{eqn_embedding_capacity_isom} of Lemma \ref{lemma_isomorphism} holds. Applying this lemma, it follows that $(M, \omega)$ and $(M', \omega')$ are isomorphic in $\Forms{kn}{k}$. Since $\mathcal{C}$ is isomorphism-closed, they are also isomorphic in $\mathcal{C}$. This proves Theorem \ref{thm_recognition}.
\end{proof}

For the proof of Lemma \ref{lemma_isomorphism} we need the following. Let $\varphi: M \to M'$ be an embedding between two topological manifolds. We denote by $\Int M$ the manifold interior of $M$ and
\begin{align*}
	&I := I_M := \{\text{connected components of }\partial M\},\\
	&I' := I_{M'},\\
	& P := M' \setminus \varphi(\mathrm{Int}(M)).
\end{align*}

For every pair of sets $X,Y$ we denote by 
\[X\sqcup Y:=\big\{(0,x)\,\big|\,x\in X\big\}\cup\big\{(1,y)\,\big|\,y\in Y\big\}\]
the disjoint union of $X$ and $Y$.\footnote{When no confusion is possible, we identify $\{0\}\times X$ with $X$ and $\{1\}\times Y$ with $Y$ to ease the notation.} We denote by
\begin{equation*}
	\psi: I \sqcup I'\enspace\to \mathcal{P}(P)
\end{equation*}
the map induced by $\varphi$ on $I$ and given by the inclusion on $I'$. We define a partition \[\mathcal{P}^{\varphi}\] on $I \sqcup I'$ by declaring that two boundary components $i, j \in  I \sqcup I'$ lie in the same element of $\mathcal{P}^{\varphi}$ if and only if there is a continuous path in $P$ that starts in $\psi(i)$ and ends in $\psi(j)$.

Let $M$ be a $kn$-dimensional (smooth)  manifold with boundary, $O$ an orientation on $M$, and $\omega$ an exact $k$-form on $M$. For every $i \in I_M$, we denote by $O_i$ the induced orientation on $i$ and by $\omega_i$ the pullback of $\omega$ to $i$.

Assume that $M$ is compact. Recall Definition \ref{def_helicity} of helicity.
\begin{definition*}[boundary helicity] We define the \emph{boundary helicity} of $(M, O, \omega)$ to be the map $h_{M, O, \omega}:I_M\to \mathbb{R}$ given by
	\begin{equation*}
		h_{M, O, \omega} (i) := h(i, O_i, \omega_i).
	\end{equation*}
\end{definition*}

\begin{lemma}[helicity inequality]\label{lemma_helicity_inequality} Let $M$ and $M'$ be compact $kn$-dimensional smooth manifolds, and  $\omega$ and $\omega'$ be maxipotent\footnote{This implies that $k$ is even.} exact $k$-forms on $M$ and $M'$, respectively. Let $C$ be a positive real number and $\varphi: M \to M'$ a smooth orientation-preserving  embedding that intertwines $C\omega$ and $\omega'$. We denote by $O$ and $O'$ the orientations of $M$ and $M'$ induced by $\omega$ and $\omega'$. Then for every $J \in \mathcal{P}^{\varphi}$, the following inequality holds:
	\begin{equation}\label{helicity_inequality}
		-C^n \sum_{i \in J \cap I}h_{M, O, \omega}(i) + \sum_{i' \in J \cap I'}h_{M', O', \omega'}(i') \geq 0.
	\end{equation}
\end{lemma}

\begin{proof}
This is Lemma 49 in \cite{JZ21}.
\end{proof}

\begin{rmk}The idea of the proof of this lemma is that for every $J \in \mathcal{P}^{\varphi}$, the left-hand side of inequality \eqref{helicity_inequality} is the volume of a path-component of $P$. This follows from Stokes' theorem for helicity (Lemma \ref{lemma_stokes_helicity}) and Remark \ref{rmk_helicity}.
	\demo
\end{rmk}

For the remainder of this section, we will use the notation \[\varphi: (M, \omega) \hookrightarrow (M', \omega')\] to denote an embedding of $M$ into $M'$ that intertwines $\omega$ and $\omega'$. A crucial ingredient of the proof of Lemma \ref{lemma_isomorphism} is the following.

\begin{lemma}[key lemma, separation of positive and negative helicity components]\label{lemma_positive_helicity_partition} Let $M$ be a nonempty, compact, 1-connected $kn$-dimensional manifold with nonempty boundary, and $\omega$ be a maxipotent exact $k$-form on $M$.  Then there exists a real number $C_0 \in (0, 1)$, such that for every $C \in (C_0, 1]$ and every embedding $\varphi: (M, C\omega)\hookrightarrow (M, \omega)$, the partition $\mathcal{P}^{\varphi}$ of $I \sqcup I$ has the following property. Every element of $\mathcal{P}^{\varphi}$ containing a negative helicity component does not contain any positive helicity component.
\end{lemma}

The proof of this lemma is based on Lemma \ref{lemma_helicity_inequality} (helicity inequality). We use that the positive helicity boundary components in the domain and the target come in pairs that lie in the same partition element and have the same helicity. (See page \pageref{proof_lemma_positive_helicity_partition}.) 

In the following, we denote
\begin{align*}I_+&:=\big\{\text{element of }I \text{ with (strictly) positive helicity}\big\},\\
I_-&:=\big\{\text{element of }I \text{ with negative helicity}\big\},\\
I_0&:=\big\{\text{element of }I \text{ with zero helicity}\big\}.
\end{align*}
We use the notation
\begin{equation*}
h:=h_{M, O, \omega}: I \to \R,\qquad h:=h_{M, O, \omega} \sqcup h_{M, O, \omega}: I \sqcup I \to \R
\end{equation*}
for both the boundary helicity map on $I$ and the induced map on the disjoint union of two copies of $I$.

\begin{proof}[Proof of Lemma \ref{lemma_isomorphism}]\label{proof_lemma_isomorphism}
	We choose $C_0$ as in the conclusion of Lemma \ref{lemma_positive_helicity_partition}, applied to $(M, \omega)$. By the inequalities in \eqref{eqn_embedding_capacity_isom}, there exist $a$ and $b$ in $(\sqrt{C_0}, 1]$ and embeddings
	\begin{equation*}
		\varphi: (M, a\omega) \hookrightarrow (M', \omega'),\qquad \psi: (M', b\omega') \hookrightarrow (M, \omega).
	\end{equation*}
	The map $\chi:=\psi \circ \varphi$ is an embedding of $M$ into itself, and it intertwines the forms $ab\omega$ and $\omega$.
	
	Since $C:=ab > C_0$, Lemma \ref{lemma_positive_helicity_partition} implies that each element of $\mathcal{P}^{\chi}$ that contains some negative helicity element does not contain any positive helicity element. Thus, summing inequality \eqref{helicity_inequality} from Lemma \ref{lemma_helicity_inequality} over all these partition elements yields
\begin{equation*}-C^n\sum_{i\in \{0\}\times I_-}h(i)+\sum_{i'\in \{1\}\times I_-}h(i') \geq 0.
\end{equation*}
Each sum is equal to $\sum_{i \in I_-}h(i)$, hence we obtain
\begin{equation}\label{eqn_negative_helicity_sum}
(1-C^n)\sum_{i \in I_-}h(i) \geq 0.
\end{equation}
Since, by assumption, $M$ has some boundary component with negative helicity, we have that $\sum_{i \in I_-}h(i) < 0$. Using inequality \eqref{eqn_negative_helicity_sum}, it follows that $C^n \geq 1$, and hence $C \geq 1$. Since $C\leq 1$ by definition, it follows that $C = ab = 1$. Because $a, b \in (0, 1]$, it follows that $a = b =1$. Hence we have
\[\chi:= \psi \circ \varphi: (M, \omega) \hookrightarrow (M, \omega).\]
	
	Since $M$ is compact, the image $\chi(M)$ is also compact and thus closed. Hence $M\setminus \chi(M)$ is open. Since $\chi$ is volume-preserving, the set $M \setminus \chi(M)$ has volume zero, which implies that it is empty. We conclude that $\chi$ and therefore $\psi$ is surjective. Hence $\psi$ is a diffeomorphism between $M$ and $M'$ that intertwines $\om$ and $\om'$. This concludes the proof of Lemma \ref{lemma_isomorphism}.
\end{proof}

In the proof of Lemma \ref{lemma_positive_helicity_partition} we will use the following. Let $M$ and $M'$ be topological manifolds and $\varphi: M \to M'$ a topological embedding.

\begin{lemma}[partition induced by an embedding]\label{lemma_property_partition} Assume that $M$ and $M'$ are compact, connected, and of the same dimension, and that $M \neq \emptyset$. Suppose also that $M'$ is 1-connected. Then for every $J \in \mathcal{P}^{\varphi}$ we have $\big|J\cap(\{0\}\x I)\big| =1$.
\end{lemma}
\begin{proof}[Proof of Lemma \ref{lemma_property_partition}]
	This follows from the proof of Lemma 47(i) in \cite{JZ21}\footnote{In this lemma, it is assumed that $\partial M' \neq \emptyset$. However, this is not necessary.}.
\end{proof}

\begin{definition}[partition element containing some boundary component]\label{defi:J i} For every $i\in I$, we denote by
\[J_i:=J_i^{\varphi}\]
the element of $\mathcal{P}^{\varphi}$ containing $(0,i)$.
\end{definition}
Since $\mathcal{P}^{\varphi}$ is a partition, for every $i \in I$, there is only one $J \in \mathcal{P}^{\varphi}$ containing it. Hence $J_i$ is well-defined.

\begin{proof}[Proof of Lemma \ref{lemma_positive_helicity_partition}]\label{proof_lemma_positive_helicity_partition}
	Set
	\begin{align}
		\begin{split}\label{def_C1}
		C_1 := \max\big\{C \in [0, 1] \,\big|\,&\exists i_+, i_+' \in I_+:\\
		&h(i_+)>h(i_+')\text{ and }h(i_+') \geq C^nh(i_+)\big\},
		\end{split}
	\end{align}
	\begin{align}
		\begin{split}\label{def_C2}
			C_2 := \max\big\{C \in [0, 1] \,\big|\,&\exists i_+\in I_+, i_-\in I_-:\\
			&(1-C^n)h(i_+) \geq -h(i_-)\},
		\end{split}
	\end{align}
\[C_0 := \max\left\{C_1, C_2\right\}.\]
Since the set $I$ is finite, these numbers are well-defined, $C_1 < 1$, $C_2 < 1$, and hence $C_0 <1$. Let $C$ be in $(0,1]$ and
\[\varphi: (M, C\omega)\hookrightarrow (M, \omega).\]
We prove Lemma \ref{lemma_positive_helicity_partition} by showing that if $C>C_0$ and $J\in \mathcal{P}^{\varphi}$ contains a positive element, then it does not contain any negative element.
	
Let $i_+ \in I_+$. By Lemma \ref{lemma_property_partition}, we have $J_{i_+}\cap \left(\{0\}\times I\right) = \left\{(0, i_+)\right\}$. Using Lemma \ref{lemma_helicity_inequality}, it follows that
	\begin{equation}\label{eqn_proof_positive_pairing_1}-C^nh(i_+)+\sum_{i\in I:\,(1,i)\in J_{i_+}}h(i)\geq0.
	\end{equation}
We define
\begin{equation}\label{eq:I i+}I_{i_+}:= \left\{i_+' \in I_+ \,\big|\, (1, i_+')\in J_{i_+}\right\}.\end{equation}
\begin{claim}\label{claim_cardinality_I_{i_+}} We have $\left|I_{i_+}\right| =1$.
\end{claim}
\begin{proof}[Proof of Claim \ref{claim_cardinality_I_{i_+}}] Since the first term on the left-hand side of \eqref{eqn_proof_positive_pairing_1} is negative, the remaining sum is positive. It follows that there exists $i_+' \in I_+$, such that $(1, i_+')\in J_{i_+}$. This shows that
		\begin{equation}\label{eqn_cardinality_positive_right}
			|I_{i_+}|\geq 1.
		\end{equation}
	Since $I_+\supseteq\bigcup_{i_+\in I_+}I_{i_+}$, we have that
	\begin{equation}\label{eqn_upperbound_union}
		\left|I_+\right|\geq\left|\bigcup_{i_+\in I_+}I_{i_+}\right|.
	\end{equation}
	Since the elements of the partition $\mathcal{P}^{\varphi}$ are disjoint, the $J_{i_+}$'s are disjoint, and hence the $I_{i_+}$'s are also disjoint. It follows that
	\begin{equation}\label{eqn_sum_equals_union}
		\left|\bigcup_{i_+\in I_+}I_{i_+}\right|=\sum_{i_+\in I_+}\left|I_{i_+}\right|.
	\end{equation}
If there exists $i_+ \in I_+$, such that $\left|I_{i_+}\right|>1$, then by \eqref{eqn_cardinality_positive_right}, we have
\begin{equation*}
\sum_{i_+\in I_+}\left|I_{i_+}\right|>\left|I_+\right|.
\end{equation*}
Combinig this with \eqref{eqn_upperbound_union} and \eqref{eqn_sum_equals_union}, it follows that there is no such $i_+$. This proves Claim \ref{claim_cardinality_I_{i_+}}.
	\end{proof}
We define the map
\[f:=f^{\varphi}:I_+ \to I_+,\,f(i_+)=i_+',\,\textrm{where }i_+'\textrm{ is the unique element of }I_{i_+}.\]
By Claim \ref{claim_cardinality_I_{i_+}} this map is well-defined. Since the $I_{i_+}$ are disjoint, $f$ is injective. Since $I_+$ is a finite set, it follows that $f$ is also surjective. Let $i'_+\in I_+$. Since $f$ is surjective, there exists $i_+\in I_+$, such that $f(i_+)=i'_+$. Using $f(i_+)\in I_{i_+}$ and \eqref{eq:I i+}, it follows that $(1,i'_+)\in J_{i_+}$. Therefore, we have
\[\left\{1\right\}\times I_+\subseteq \bigcup_{i_+\in I_+}J_{i_+}.\]
By Definition \ref{defi:J i} we have $(0,i_+)\in J_{i_+}$, for every $i_+ \in I_+$. It follows that
\begin{equation}\label{eqn_covering_positive_elements}
\{0,1\}\x I_+\subseteq\bigcup_{i_+\in I_+}J_{i_+}.
\end{equation}
Let $i_+ \in I_+$. We define
\[K_{i_+}:=K_{i_+}^{\varphi}:=\big\{i\in I\wo\{f(i_+)\}\,\big|\,(1,i)\in J_{i_+}\big\}.\]
Let $C \in (C_1, 1]$ and $\varphi:(M, C\omega) \hookrightarrow (M, \omega)$.
	\begin{claim}\label{claim_K_nonnegative}
		If $C>C_2$, then $K_{i_+}\cap I_-=\emptyset$.
	\end{claim}
	\begin{proof}[Proof of Claim \ref{claim_K_nonnegative}]
It follows from the definitions of $K_{i_+}$ and $f$ that the elements of $K_{i_+}$ have negative or zero helicity. 
\begin{claim}\label{claim_composition_hf}We have
\[h(f(i_+))=h(i_+).\]
\end{claim}
\begin{proof}[Proof of Claim \ref{claim_composition_hf}] Using \eqref{eqn_proof_positive_pairing_1} and that the elements of $K_{i_+}$ have negative or zero helicity, we have
			\begin{equation*}
				-C^nh(i_+) + h(f(i_+)) \geq 0.
			\end{equation*}
			Considering $i_+':=f(i_+)$ and using our assumption that $C>C_1$, by the definition of $C_1$ in \eqref{def_C1}, it follows that
			\begin{equation}\label{eqn_leq_hf}
				h(i_+)\leq h(f(i_+)).
			\end{equation}
			Since $f$ is a permutation on the finite set $I_+$, there exists $s_0 \in \N$, such that $f^{s_0}(i_+)=i_+$.\footnote{Here $f^{s_0}$ denotes the composition $f\circ\cdots\circ f$, in which $f$ occurs $s_0$ times.} From this and \eqref{eqn_leq_hf}, it follows that
			\begin{equation*}
				h(i_+) \leq h(f(i_+)) \leq \dots \leq h(f^{s_0}(i_+))=h(i_+).
			\end{equation*}
			This implies that $h(i_+)=h(f(i_+))$. This concludes the proof of Claim \ref{claim_composition_hf}.
		\end{proof}
Assume that $K_{i_+}\cap I_-\neq\emptyset$. We choose $i_-\in K_{i_+}\cap I_-$. We have
\begin{align*}(1-C^n)h(i_+)&=-C^nh(i_+)+h(f(i_+))\qquad\textrm{(using Claim \ref{claim_composition_hf})}\\
&\geq-\sum_{i\in K_{i_+}}h(i)\qquad\textrm{(using \eqref{eqn_proof_positive_pairing_1})}\\
&\geq-h(i_-)\qquad\textrm{(since $h(i)\leq0$, for every $i\in K_{i_+}$).}
\end{align*}
Using the definition \eqref{def_C2} of $C_2$, it follows that $C\leq C_2$. The statement of Claim \ref{claim_K_nonnegative} follows.
\end{proof}	 
Let $C\in(C_0,1]$ and $\varphi:(M,C\omega)\hookrightarrow(M,\omega)$. Let $J$ be an element of $\mathcal{P}^{\varphi}$ that contains some $i\in\{0,1\}\x I_+$. By \eqref{eqn_covering_positive_elements}, there exists $i_+\in I_+$, such that $i\in J_{i_+}$. Since the elements of $\mathcal{P}^{\varphi}$ are disjoint, it follows that $J=J_{i_+}$. Hence by Lemma \ref{lemma_property_partition} and Claim \ref{claim_K_nonnegative}, $J$ does not contain any negative helicity element. The statement of Lemma \ref{lemma_positive_helicity_partition} follows.
\end{proof}

\section{Proof of Proposition \ref{prop_recognition_normalized_capacities}}\label{section_proof_prop}

In the proof of Proposition \ref{prop_recognition_normalized_capacities} we will use the following.
\begin{remark}[normalization of capacity]\label{rmk_def_cbar} Let $\mathcal{C}=(\mathcal{O}, \mathcal{M})$ be a $(2n, 2)$-form category containing the open unit ball $B$ and cylinder $Z$. Let $\mathcal{C}_0$ be the full subcategory of $\mathcal{C}$ whose objects are given by $\mathcal{O}_0:=\mathcal{O}\cap \mathcal{O}_0^{2n, 2}$, where $\mathcal{O}_0^{2n, 2}$ is defined in \eqref{eq_subcat_beth1}. Let $(M, \omega)$ be an object of $\Forms{2n}{2}$. We define
	\begin{equation}\label{eqn_def_cbar}
		\bar{c}_{(M, \omega)}:= \max\{c_{(M, \omega)}, w\}.
	\end{equation}
	Then $(\mathcal{C}_0, \bar{c}_{(M, \omega)})$ is a capacity on $\mathcal{C}$. Furthermore, if
	\begin{equation}\label{eqn_condition_normalization}
		c_{(M, \omega)}(Z) \leq 1,\footnote{Compare to footnote \ref{evaluate}.}
	\end{equation}
	then $(\mathcal{C}_0, \bar{c}_{(M, \omega)})$ is a normalized capacity on $\mathcal{C}$.
	\demo
\end{remark}

\begin{proof}[Proof of Proposition \ref{prop_recognition_normalized_capacities}] \setcounter{claim}{0}
	\begin{claim}\label{claim_normalized_cap_isom} Let $(M,\omega)$ and $(M', \omega')$ be objects of $\mathcal{C}$. Assume that all normalized capacities, whose domain contains $(M, \omega)$ and $(M',\omega')$, agree on $(M, \omega)$ and $(M', \omega')$. Suppose also that $(M, \omega)$ and $(M', \omega')$ are not isomorphic to $B$ nor $Z$. Then they are isomorphic to each other in $\mathcal{C}$.
	\end{claim}
	\begin{proof}[Proof of Claim \ref{claim_normalized_cap_isom}]
	We define
	\begin{equation*}
		A:=c_{(M, \omega)}(Z).
	\end{equation*}
Since $M$ has a negative helicity boundary component, it is nonempty. Hence there exists $a\in(0,\infty)$ and a ball $B'$ of area $a$ that symplectically embeds into $(M,\omega)$. We have
\begin{equation}\label{eq:A c M om}A=c_{(M,\omega)}(Z)\leq c_{B'}(Z)=\frac\pi a<\infty.\end{equation}
\begin{claim}\label{claim:c M om M' om'} We have
\begin{equation}\label{condition_lemma}
c_{(M,\omega)}(M',\omega')\geq1.
\end{equation}
\end{claim}
\begin{proof}[Proof of Claim \ref{claim:c M om M' om'}] \textbf{Case} $A \neq 0$: Then \eqref{eqn_condition_normalization} with $(M,\om)$ replaced by $(M, A\omega)$ holds. (Here we use \eqref{eq:A c M om}, in order to make sense of $A\om$.) Hence by Remark \ref{rmk_def_cbar} the map $\bar{c}_{(M, A\omega)}$ is a normalized capacity. Hence by our assumption we have that
\begin{align}\nn\bar{c}_{(M, A\omega)}(M',\omega')&=\bar{c}_{(M, A\omega)}(M,\omega)\\
\label{eqn_equality_normalized_cap}&\geq\frac1A,
\end{align}
where in the second step we used (\ref{eqn_def_cbar},\ref{eqn_def_embedding_capacity}) and our assumption $A\neq 0$. Since the Gromov width $w$ is a normalized capacity, we have by assumption that
	\begin{align*}
		w(M',\omega') &=w(M,\omega)\\
		&<\frac1A,\tag*{(by our assumption \eqref{ineq_thin_object}).}
	\end{align*}
Combining this with \eqref{eqn_def_cbar} and \eqref{eqn_equality_normalized_cap}, it follows that
	\begin{equation*}
		c_{(M,A\omega)}(M',\omega')\geq\frac1A.
	\end{equation*}
Inequality \eqref{condition_lemma} follows.\\

\noindent\textbf{Case} $A=0$: We define
	\begin{equation*}
		\varepsilon := \frac{1}{2w(M, \omega)}.
	\end{equation*}
Since $M$ is nonempty, we have $w(M,\omega)>0$, which implies that $\epsilon < \infty$. Hence $\epsilon\omega$ is well-defined. Since $M$ is compact, we also have that $w(M, \omega)<\infty$ and hence $\epsilon >0$. Since $c_{(M,\omega)}(Z)=A=0$ and $\epsilon>0$, inequality \eqref{eqn_condition_normalization} holds with $(M,\om)$ replaced by $(M,\varepsilon\omega)$. Hence by Remark \ref{rmk_def_cbar} the map $\bar{c}_{(M, \varepsilon\omega)}$ is a normalized capacity. Using our assumption, it follows that
\begin{align}
		\begin{split}\label{eqn_equality_normalized_cap_case2}
			\bar{c}_{(M, \varepsilon\omega)}(M',\omega')&=\bar{c}_{(M, \varepsilon\omega)}(M,\omega)\\
			& \overset{\text{by }(\ref{eqn_def_cbar},\ref{eqn_def_embedding_capacity})}{\geq}\frac1\eps.
		\end{split}
	\end{align}
	Since $w$ is a normalized capacity, by assumption, we have
	\begin{align*}
		w(M',\omega')&=w(M,\omega)\\
		&=\frac{1}{2\eps}\tag*{(by definition of $\varepsilon$).}
	\end{align*}
	Using \eqref{eqn_equality_normalized_cap_case2}, it follows that
	\begin{equation*}
		c_{(M, \varepsilon\omega)}(M',\omega')\geq\frac1\eps.
	\end{equation*}
Since $\varepsilon>0$, inequality \eqref{condition_lemma} follows.\\

We have shown that this inequality holds in all cases. This proves Claim \ref{claim:c M om M' om'}.
\end{proof}
Claim \ref{claim:c M om M' om'} with the r\^oles of $(M, \omega)$ and $(M', \omega')$ interchanged, implies that
	\begin{equation*}
		c_{(M', \omega')}(M, \omega) \geq 1.
	\end{equation*}
	Combining this with Claim \ref{claim:c M om M' om'} and Lemma \ref{lemma_isomorphism}, it follows that $(M, \omega)$ and $(M', \omega')$ are isomorphic in $\Forms{2n}{2}$. (Here we used our assumption that these objects are not isomorphic to $B$ or $Z$, and therefore they satisfy the hypotheses of Theorem \ref{thm_recognition}.) Since $\mathcal{C}$ is isomorphism-closed, it follows that $(M, \omega)$ and $(M', \omega')$ are isomorphic in $\mathcal{C}$. This proves Claim \ref{claim_normalized_cap_isom}.
	\end{proof}
	
\begin{claim}\label{claim_existence_normalized_cap} Let $(M,\omega)$ be an object in $\mathcal{C}$ that is not isomorphic to the ball nor the cylinder, such that $w(M,\om)=1$. Then there exists a normalized capacity on $\mathcal{C}$ that distinguishes $(M, \omega)$ from $B$ and $Z$.
	\end{claim}
	\begin{proof}[Proof of Claim \ref{claim_existence_normalized_cap}] We define $A:=c_{(M, \omega)}(Z)$. As in the proof of Claim \ref{claim_normalized_cap_isom}, we have $A<\infty$.\\

\noindent\textbf{Case} $A \neq 0$: By our assumption that $(M,\om)$ is not isomorphic to $B$ nor $Z$, inequality \eqref{ineq_thin_object} holds. We have
\begin{align*}
\bar{c}_{(M, A\omega)}(M, \omega) &\geq \frac{1}{A}\tag*{(using \eqref{eqn_def_cbar} and $A\neq0$)}\\
&> 1 \tag*{(using \eqref{ineq_thin_object} and our assumption that $w(M,\omega)=1$)}\\
&=\bar{c}_{(M, A\omega)}(Z)\tag*{(by our definition of $A$).}
\end{align*}
It follows that $\bar{c}_{(M, A\omega)}$ distinguishes $(M, \omega)$ from $B$ and $Z$.\\

\noindent\textbf{Case} $A=0$: Then we have
\[\bar{c}_{(M, \frac{1}{2}\omega)}(M, \omega)\geq2>1=\bar{c}_{(M, \frac{1}{2}\omega)}(Z),\]
where in the last step we used that $A=0$ and our assumption $w(M,\om)=1$. It follows that $\bar{c}_{(M, \frac{1}{2}\omega)}$ distinguishes $(M, \omega)$ from $B$ and $Z$. This proves Claim \ref{claim_existence_normalized_cap}.
\end{proof}
The statement of Proposition \ref{prop_recognition_normalized_capacities} follows from Claims \ref{claim_normalized_cap_isom} and \ref{claim_existence_normalized_cap} and the fact that in the case $w(M,\om)\neq1$, the Gromov width distinguishes $\smfld$ from $B$ and $Z$.
\end{proof}

\section{Hypotheses of the main result that cannot be dropped}\label{sec:nec}

The hypothesis in Theorem \ref{thm_recognition} that every object of $\mathcal{C}$ has some negative helicity boundary component cannot be dropped. This is the content of the following proposition.

\begin{prop}[negative helicity hypothesis cannot be dropped]\label{prop:neg hel} Let $n\geq2$ and consider the full subcategory $\mathcal{C}$ of $\Om^{2n,2}$ whose objects are the compact, 1-connected, exact symplectic manifolds. This category does not satisfy the conclusion of Theorem \ref{thm_recognition}.\footnote{The objects of $\mathcal{C}$ need not satisfy the helicity assumption of Theorem \ref{thm_recognition}.}
\end{prop}
\begin{proof}[Proof of Proposition \ref{prop:neg hel}]\setcounter{claim}{0} By Theorem 1.1 in \cite{EH96} there exist strongly starshaped\footnote{We call a submanifold $M$ of $\R^m$ \emph{strongly starshaped} iff for every point $x\in M$ the line segment from 0 to $x$ is contained in $M$ and every line emanating from 0 intersects the boundary of $M$ transversely.} compact submanifolds $M$ and $M'$ of $\R^{2n}$ of dimension $2n$ that are not symplectomorphic w.r.t.~the standard symplectic form $\omst$, but have symplectomorphic interiors.

The manifolds $M$ and $M'$ are contractible and therefore 1-connected.\footnote{We may choose $M$ and $M'$ to be arbitrarily $C^\infty$-close to the closed unit ball.} It follows that $(M,\omst)$ and $(M',\omst)$ are objects of $\mathcal{C}$.
\begin{claim}\label{claim:dist}The capacities on $\mathcal{C}$ do not distinguish these objects.
\end{claim}
\begin{pf}[Proof of Claim \ref{claim:dist}] Let $c$ be a capacity on $\mathcal{C}$. Let $a\in(0,1)$. We choose a symplectomorphism $\phi:\Int M\to\Int M'$. Since $M$ is strongly starshaped, $aM$ is contained in $\Int M$. Hence $\phi(aM)$ is well-defined. We have
\begin{align*}&a^2c(M,\omst)\\
=&c\big(M,a^2\omst\big)\\
=&c\big(aM,\omst\big)\\
&\textrm{(since $\big(M,a^2\omst\big)$ and $\big(aM,\omst\big)$ are symplectomorphic)}\\
=&c\big(\phi(aM),\omst\big)\,\textrm{(using that $\phi:aM\to\phi(aM)$ is a symplectomorphism)}\\
\leq&c(M',\omst)\qquad\textrm{(using that $\phi(aM)\sub\Int M'\sub M'$).}
\end{align*}
Since this holds for every $a\in(0,1)$, it follows that $c(M,\omst)\leq c(M',\omst)$. Interchanging the r\^oles of $M$ and $M'$, the opposite inequality follows. Hence we have
\[c(M,\omst)=c(M',\omst).\]
Since this holds for all capacities $c$, it follows that $(M,\omst)$ and $(M',\omst)$
cannot be distinguished by capacities. This proves Claim \ref{claim:dist} and completes the proof of Proposition \ref{prop:neg hel}.
\end{pf}
\end{proof}

\begin{remark}[compactness hypothesis cannot be dropped]\label{rmk:cpt} The compactness hypothesis in Theorem \ref{thm_recognition} cannot be dropped. To see this, we denote by $B^m_r(x)$ (resp. $\overline{B}^m_r(x))$ the open (resp.~closed) ball of radius $r$ around $x$ in $\mathbb{R}^{m}$. We define
\begin{align*}
&M:=\\
&\overline{B}^{2n}_3\big(0, -2, 0, \dots, 0\big) \setminus\Big(B^{2n}_1\big(0, -2, 0, \dots, 0\big) \cup\left([0, 1]\times\left\{\big(0, \dots, 0\big)\right\}\right)\Big),\\
&M':= M \setminus\left\{\big(2, 0, \dots, 0\big)\right\}.
\end{align*}
See Figure \ref{fig_slit_in_shell}. We equip $M$ and $M'$ with the standard symplectic form. Since $M' \subseteq M$, it follows that $c(M') \leq c(M)$ for every capacity $c$. We consider the Hamiltonian function $H\big(q^1, p_1, \dots, q^n, p_n\big):= q^1p_1$. We choose a smooth function $\rho: \mathbb{R}^{2n} \to \mathbb{R}$ that equals $1$ in some neighborhood of $[0,2] \times\left\{\big(0, \dots, 0\big)\right\}$ and has compact support in $\overline{B}_3\big(0, -2, 0, \dots, 0\big) \setminus B_1\big(0, -2, 0, \dots, 0\big)$. The Hamiltonian flow of $H$ in $\R^{2n}$ is given by
\begin{equation*}
	\varphi_H^t\big(q^1, p_1, \dots q^n, p_n\big) = \big(e^tq^1, e^{-t}p_1, q^2, p_2, \dots, q^n, p_n\big).
\end{equation*}

It follows that the restriction of the Hamiltonian time-$t$ flow $\varphi^t$ of $\rho H$ to $[0, 1] \times \{(0, \dots, 0)\}$ is given by the same formula. Hence the image of $[0, 1]\times\left\{\big(0, \dots, 0\big)\right\}$ under $\varphi^{\log(2)}$ includes $[0,2]\times\{(0,\dots,0)\}$. Therefore, $\varphi^{\log(2)}$ maps $M$ to $M'$. See Figure \ref{fig_slit_in_shell}. It follows that $c(M) \leq c(M')$ for every capacity $c$. Hence all capacities agree on $M$ and $M'$, although these manifolds are not diffeomorphic (and hence not symplectomorphic).
\demo
\end{remark}

\begin{figure}[ht]
	\centering
	
	\def\svgwidth{0.8\columnwidth}
	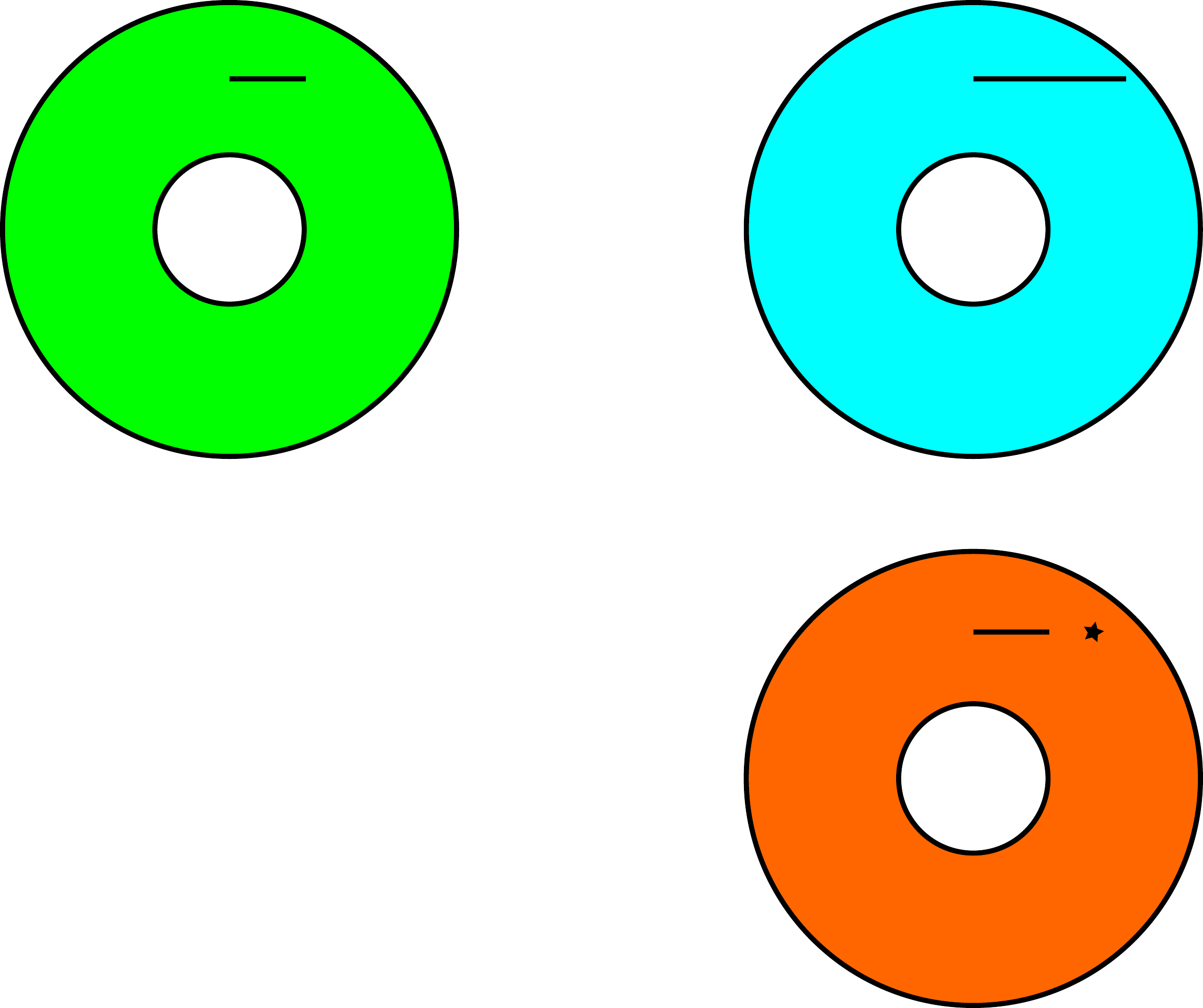

	\caption{The green slit shell is $M$, and the orange punctured and slit shell is $M'$, which is $M$ with a point removed. The blue shell is the image of $M$ under the Hamiltonian flow $\phi^{\log2}$. The line segment is stretched so much that it includes the puncture of the orange shell. Hence the orange punctured and slit shell includes the blue slit shell. (For better visibility we moved the puncture slightly to the left.)}
	\label{fig_slit_in_shell}
\end{figure}

\appendix
\section{Isomorphism-closed subcategories}\label{section_subcat}
In this section, we address some set-theoretic issues that arise when considering categories and subcategories in the ZFC axiomatic system. In ZFC, there is no notion of a class. However, ZFC can deal with classes that are determined by a well-formed logical formula, such as the class of all sets, and the class of all smooth manifolds. This way, every statement about the class can be reformulated in terms of the logical formula. (Since there are only countably many logical formulas, they do not define all categories.)

Let $\widehat{\mathcal{C}}$ be a category determined by a formula, and let $\mathcal{C}_0$ be a small subcategory of $\widehat{\mathcal{C}}$, i.e. ~a pair $(\mathcal{O}_0, \mathcal{M}_0)$ of sets that satisfies the logical formula defining $\widehat{\mathcal{C}}$.

\begin{definition}[isomorphism-closure]\label{def_isom_closure} We define the \emph{isomorphism-closure} of $\mathcal{C}_0$ in $\widehat{\mathcal{C}}$ to be the category whose objects are the objects of $\widehat{\mathcal{C}}$ that are isomorphic in $\widehat{\mathcal{C}}$ to an object of $\mathcal{C}_0$, and whose morphisms are obtained by composition of the isomorphisms of $\widehat{\mathcal{C}}$ with the morphisms of $\mathcal{C}_0$. More specifically, the morphisms are of the form $\psi \circ f_0 \circ \varphi^{-1}$, where $\varphi$ and $\psi$ are isomorphisms in $\widehat{\mathcal{C}}$ and $f_0$ is a morphism in $\mathcal{C}_0$.
\end{definition}

Assume that $\widehat{\mathcal{C}}$ is locally small and that there exists a set $S$ of objects of $\widehat{\mathcal{C}}$, such that every object of $\widehat{\mathcal{C}}$ is isomorphic to an element of $S$.

\begin{definition}[isomorphism-closed subcategory]\label{def_isom_closed_subcat} The subcategories of $\widehat{\mathcal{C}}$ that we obtain as isomorphism-closures as in Definition \ref{def_isom_closure} are called \emph{isomorphism-closed subcategories} of $\widehat{\mathcal{C}}$.
\end{definition}

\begin{remark}[isomorphism-closed subcategory]\label{rmk_isom_closed_subcat}
	The category $\widehat{\mathcal{C}}$ is an isomorphism-closed subcategory of itself. This holds, because the full subcategory of $\widehat{\mathcal{C}}$ with the elements of $S$ as objects is small by the local smallness of $\widehat{\mathcal{C}}$. This is the reason for imposing the conditions of the existence of $S$ and the local smallness of $\widehat{\mathcal{C}}$. Without these conditions $\widehat{\mathcal{C}}$ need not be an isomorphism-closed subcategory of itself in the sense of Definition \ref{def_isom_closed_subcat}. As an example, the category with sets as objects and only the identity morphisms as morphisms is not an isomorphism-closed subcategory of itself in this sense.
\demo
\end{remark}
\begin{rmk}[essential bijection with a small subcategory]
	 A functor $F: \mathcal{C} \to \mathcal{D}$ is called \emph{essentially injective} iff for every pair $(A,B)$ of objects in $\mathcal{C}$, the condition that $F(A)$ is isomorphic to $F(B)$ implies that $A$ is isomorphic to $B$. The functor $F$ is \emph{essentially surjective} iff for every object $B$ of $\mathcal{D}$, there is an object $A$ of $\mathcal{C}$, such that $B$ is isomorphic to $F(A)$. A functor that is both essentially injective and essentially surjective is called \emph{essentially bijective}. Let $S$ be as described before Definition \ref{def_isom_closed_subcat}. The inclusion of the full small subcategory of $\widehat{\mathcal{C}}$ whose objects are the elements of $S$ into $\widehat{\mathcal{C}}$ is an essentially bijective functor.
	\demo
\end{rmk}

Let $G$ be a group that is determined by a logical formula, and let $(\widehat{\mathcal{O}}, \widehat{\mathcal{M}}):=\widehat{\mathcal{C}}$ be a category determined by a formula. Formally, by a $G$-action on $\widehat{\mathcal{C}}$ we mean a map
\begin{equation*}
	\rho: G \times \widehat{\mathcal{C}} \to \widehat{\mathcal{C}}
\end{equation*}
such that for every $g, h$ in $G$, every $A$ in $\widehat{\mathcal{O}}$ and every $f$ in $\widehat{\mathcal{M}}$, we have that
\begin{align*}
	\rho_g:=\rho(g, \cdot) &\text{ is a functor }\widehat{\mathcal{C}} \to \widehat{\mathcal{C}}\\
	\rho_e(A) = A &\quad \rho_e(f) = f\\
	\rho_h(\rho_g(A)) = \rho_{hg}(A) &\quad \rho_h(\rho_g(f)) = \rho_{hg}(f),
\end{align*}
where $e$ denotes the neutral element of $G$. The following definition makes this notion precise within ZFC.
\begin{definition}[group action on a category]\label{def_action_cat} A \emph{$G$-action} on $\widehat{\mathcal{C}}$ is a logical formula that satifies the properties corresponding to the above conditions.
\end{definition}
\begin{rmk}[group action on a category] We think of a $G$-action $\rho$ as the ``subclass'' of $(G\x\widehat{\mathcal{C}})\x\widehat{\mathcal{C}}$ determined by $\rho$.
\end{rmk}

\begin{definition}[invariant subcategory]\label{def_invariant_subcat} An isomorphism-closed or small subcategory $\mathcal{C}$ of $\widehat{\mathcal{C}}$ is called $G$-\emph{invariant} iff for every $g \in G$, the functor $\rho_g$ leaves $\mathcal{C}$ invariant.
\end{definition}

\begin{remark}[generating invariant small subcategory]\label{rmk_invariant_subcat} If $\mathcal{C}$ is an invariant isomorphism-closed subcategory of $\widehat{\mathcal{C}}$, and if $\widehat{\mathcal{C}}$ is locally small, then there exists an invariant small subcategory $\mathcal{C}_0$ of $\widehat{\mathcal{C}}$, such that $\mathcal{C}$ is the isomorphism-closure of $\mathcal{C}_0$. Indeed, if $\widetilde{\mathcal{C}_0}=(\widetilde{\mathcal{O}}_0, \widetilde{\mathcal{M}}_0)$ is a small subcategory whose isomorphism-closure is $\mathcal{C}$, then the small category $\mathcal{C}_0 =(\mathcal{O}_0, \mathcal{M}_0)$ given by
	\begin{equation*}
		\mathcal{O}_0 = \big\{\rho_g(A)\big| g \in G, A \in \widetilde{\mathcal{O}}_0\big\} \text{\quad and \quad}\mathcal{M}_0 = \big\{\rho_g(f)\big| g \in G, f \in \widetilde{\mathcal{M}}_0\big\}
		\end{equation*}
	is invariant and has $\mathcal{C}$ as isomorphism-closure.
\end{remark}

\section{Helicity}\label{section_helicity}

In this section, we list some properties of helicity and prove a version of Stokes' Theorem for helicity. This is the main ingredient of the proof of Theorem \ref{thm_recognition}, used indirectly through Lemma \ref{lemma_helicity_inequality}. Consider the setting of Definition \ref{def_helicity} (helicity).
\begin{lemma}\label{lemma_helicity_well_def} Helicity is well-defined, i.e. it is independent of the choice of a primitive.
\end{lemma}
\begin{proof}[Proof of Lemma \ref{lemma_helicity_well_def}]
	Let $\alpha$ and $\alpha'$ be two primitives of $\sigma$. Then $\alpha - \alpha'$ is closed and hence $(-1)^{k-1}(\alpha-\alpha')\wedge\alpha\wedge\sigma^{\wedge(n-2)}$ is a primitive of $(\alpha - \alpha') \wedge \sigma^{\wedge(n-1)}$. Using that Stokes' Theorem and our assumption that $N$ does not have boundary, it follows that
\[\int_{N}(\alpha-\alpha')\wedge\sigma^{\wedge(n-1)}=0.\]
The statement of Lemma \ref{lemma_helicity_well_def} follows.	
\end{proof}
\begin{rmk}[well-definedness of helicity, odd degree] \begin{itemize}
\item Helicity is only well-defined for $n \geq 2$. See \cite[Remark 31]{JZ21}.
\item It vanishes if $k$ is odd. See \cite[Remark 31]{JZ21}.
\end{itemize}
\end{rmk}
The following facts about helicity are used in the proof of Lemma \ref{lemma_helicity_inequality}, which is Lemma 49 in \cite{JZ21}.
\begin{Rmks}[helicity]\label{rmk_helicity}
\begin{itemize}		
		\item For every $C \in \mathbb{R}$, we have $h(N, O, C\sigma) = C^n h(N, O, \sigma)$.
		\item Reversing the orientation on $N$ changes the sign of the helicity: $h(N, \overline{O}, \sigma) = -h(N, O, \sigma)$, where $\overline{O}$ is the orientation opposite to $O$.
		
		\item Let $N$ and $N'$ be two closed $kn-1$-dimensional manifolds, $O$ an orientation on $N$, $\sigma$ an exact $k$-form on $N$, and $\varphi: N \to N'$ a smooth embedding. Then we have
		\begin{equation*}
			h(\varphi(N), \phi_*O, \varphi_*\sigma) = h(N, O, \sigma).
		\end{equation*}
	\end{itemize}
	\demo
\end{Rmks}

The next lemma is one of the main ingredients in the proofs of the main results. It says that the helicity of the boundary of a manifold is equal to the volume of that manifold. Let $M$ be a manifold, $N \subseteq M$ a submanifold and $\omega$ a differential form on $M$. We define
\begin{equation*}
	\omega_N := \text{ pullback of }\omega \text{ by the inclusion of }N \text{ into }M.
\end{equation*}
Let $\partial M$ denote the boundary of $M$ and assume that $N \subseteq \partial M $. Let $O$ be an orientation on $M$. We define:
\begin{equation*}
	O_N := \text{ orientation of }N \text{ induced by }O.
\end{equation*}

Let $k,n\in\N$, such that $n\geq 2$. Let $(M, O)$ be a compact, oriented manifold of dimension $kn$, and $\omega$ be an exact $k$-form on $M$.

\begin{lemma}[Stokes' Theorem for helicity]\label{lemma_stokes_helicity}
	The following equation holds:
	\begin{equation*}
		\int_{M, O} \omega^{\wedge n} =h(\partial M, O_{\partial M}, \omega_{\partial M}).
	\end{equation*}
\end{lemma}
\begin{proof}[Proof of Lemma \ref{lemma_stokes_helicity}]
	Let $\alpha$ be a primitive of $\omega$. Then we have that $\omega^{\wedge n} = d(\alpha \wedge \omega^{\wedge (n-1)})$. Using Stokes' Theorem, we obtain
	\begin{equation*}
		\int_{M, O}\omega^{\wedge n} = \int_{\partial M, O_{\partial M}} \alpha \wedge \omega^{\wedge (n-1)} = h(\partial M, O_{\partial M}, \omega_{\partial M}).
	\end{equation*}
	This completes the proof.
\end{proof}

\begin{rmks}[Stokes' Theorem for helicity]\begin{itemize}\item If $k$ is odd then both sides of the equality in the statement of Lemma \ref{lemma_stokes_helicity} are zero. Here for the right hand side see \cite[Remark 31]{JZ21}.
\item If $\omega$ is maxipotent and $O$ is the orientation induced by $\omega$, then Lemma \ref{lemma_stokes_helicity} implies that the helicity of the boundary of $M$ with respect to the orientation induced by $O$ is positive, as it equals the volume of $M$.
\item If the image of a negative helicity boundary component under a form-preserving embedding bounds a compact manifold (e.g.~a ball), then the volume of this manifold equals minus the helicity of the boundary component. This means that the negative helicity component eats up some volume. Compare to the left hand side of Figure \ref{fig_embeddings} on page \pageref{fig_embeddings}. We used this in the proofs of our main results.
\end{itemize}
\end{rmks}

\bibliographystyle{amsalpha}
\bibliography{mathematicalBibliography}

\providecommand{\bysame}{\leavevmode\hbox to3em{\hrulefill}\thinspace}
\providecommand{\MR}{\relax\ifhmode\unskip\space\fi MR }
% \MRhref is called by the amsart/book/proc definition of \MR.
\providecommand{\MRhref}[2]{%
  \href{http://www.ams.org/mathscinet-getitem?mr=#1}{#2}
}
\providecommand{\href}[2]{#2}
\begin{thebibliography}{CHLS07}

\bibitem[CHLS07]{CHLS07}
Kai Cieliebak, Helmut Hofer, Janko Latschev, and Felix Schlenk,
  \emph{Quantitative symplectic geometry}, Dynamics, ergodic theory, and
  geometry, Math. Sci. Res. Inst. Publ., vol.~54, Cambridge Univ. Press,
  Cambridge, 2007, pp.~1--44. \MR{2369441}

\bibitem[EH96]{EH96}
Yakov Eliashberg and Helmut Hofer, \emph{Unseen symplectic boundaries},
  Manifolds and geometry ({P}isa, 1993), Sympos. Math., XXXVI, Cambridge Univ.
  Press, Cambridge, 1996, pp.~178--189. \MR{1410072}

\bibitem[Gro85]{Gro85}
Mikhail Gromov, \emph{Pseudo holomorphic curves in symplectic manifolds},
  Inventiones mathematicae \textbf{82} (1985), no.~2, 307--347.

\bibitem[JZ21]{JZ21}
Du\v{s}an Joksimovi\'c and Fabian Ziltener, \emph{Generating sets and
  representability for symplectic capacities}, 2021, arXiv:2003.06442, accepted
  by J. Symplectic Geom.

\bibitem[KL21]{KL}
Ely Kerman and Yuanpu Liang, \emph{On symplectic capacities and their blind
  spots}, 2021, arXiv:2109.01792, to appear in the Journal of Topology and
  Analysis.

\bibitem[SZ12]{SZ12}
Jan Swoboda and Fabian Ziltener, \emph{Coisotropic displacement and small
  subsets of a symplectic manifold}, Mathematische Zeitschrift \textbf{271}
  (2012), no.~1, 415--445.

\end{thebibliography}

\end{document}